\documentclass[12pt]{amsart}

\newtheorem{theorem}{Theorem}[section]
\newtheorem{proposition}[theorem]{Proposition}
\newtheorem{lemma}[theorem]{Lemma}

\theoremstyle{remark}
\newtheorem{remark}[theorem]{Remark}
\newtheorem{definition}[theorem]{Definition}

\newcommand{\cU}{\mathcal{U}}
\newcommand{\cW}{\mathcal{W}}
\newcommand{\R}{\mathbb{R}}
\newcommand{\C}{\mathbb{C}}
\newcommand{\Real}{\textrm{\rm Re}\,}
\newcommand{\Imag}{\textrm{\rm Im}\,}
\newcommand{\diag}{\textrm{\rm diag}\,}
\newcommand{\diff}{\textrm{\rm d}}
\newcommand{\sgn}{\textrm{\rm sgn}\,}

\begin{document}

\thispagestyle{empty}

\begin{center}
{\sf\Large Spectral stability of weak}\medskip

{\sf\Large relaxation shock profiles}\bigskip

{\tt \today}
\end{center}

\begin{equation*}
\begin{array}{cc}
\textsf{\large Corrado MASCIA}  & \textsf{\large Kevin ZUMBRUN}\\
\textsf{\footnotesize Dipartimento di Matematica ``G. Castelnuovo''}\qquad
&\textsf{\footnotesize Mathematics Department}\\
\textsf{\footnotesize Sapienza -- Universit\`a di Roma}\qquad
&\textsf{\footnotesize Indiana University}\\
\textsf{\footnotesize P.le Aldo Moro, 2 - 00185 Roma (ITALY)}\qquad
&\textsf{\footnotesize Bloomington, IN  47405-4301 (USA)}
\end{array}
\end{equation*}
\bigskip

\hrule\bigskip

\begin{quote}
{\footnotesize
{\sf Abstract:} 
Using a combination of Kawashima- and Goodman-type energy estimates,
we establish spectral stability of general small-amplitude relaxation shocks 
of symmetric dissipative systems.
This extends previous results obtained
by Plaza and Zumbrun \cite{PlZu04}
by singular perturbation techniques under an
additional technical assumption, namely, that the background equation 
be noncharacteristic with respect to the shock.
}
\end{quote}
\bigskip

\hrule\bigskip

\baselineskip=18pt

\section{Introduction}

Let us consider the one-dimensional hyperbolic system with relaxation
\begin{equation}\label{relax}
   w_t+F(w)_x=Q(w)
\end{equation}
for the unknown $w=w(x,t)\in\R^N$, $x\in\R, t>0$.
Here $F\in C^2(\R^N\,;\,\R^N)$ is such that $\diff F(w)$ has $N$ real distinct
eigenvalues for any state $w$ under consideration and $Q\in C^1(\R^N\,;\,\R^N)$ 
has the structure $Q(w)=(0_n,q(w))$ where $q\in C^1(\R^N\,;\,\R^r)$, $r=N-n$.
Additionally, we assume the function $q$ to have a relaxation structure:
let $w=(u,v)\in{\mathcal U}\times {\mathcal V}\subseteq \R^n\times\R^r$,\\
{\sf i.} there exists a $C^1$ function $v^*\,:\,{\mathcal U}\subset\R^n\to\R^r$
such that $q(w)=0$ in ${\mathcal U}\times {\mathcal V}$
if and only if $v=v^*(u)$ where $w=(u,v)\in\R^n\times\R^r$;\\
{\sf ii.} for any $u\in{\mathcal U}$, all of the eigenvalues of $\diff_u q(u,v^*(u))$ 
have negative real part.

As a consequence, it is natural to introduce the corresponding {\sf relaxed} hyperbolic 
system of conservation laws, formally obtained by considering the first $n$ equations
of (\ref{relax}) and substituting the variable $v$ with the equilibrium $v=v^*(u)$
\begin{equation}\label{relaxed}
 u_t+f^*(u)_x=0
 \qquad\qquad\textrm{where}\quad 
 f^*(u):=f(u,v^*(u)).
\end{equation}
System (\ref{relax}) possesses smooth traveling wave solutions 
corresponding to shock waves of the relaxed system (\ref{relaxed}) 
at least in the small--amplitude case.\footnote{In all of the paper, 
we will not  consider the case of  nonsmooth relaxation shock 
profiles, i.e. exhibiting subshocks. This is not restrictive since
the smallness assumption of the profile, usually, guarantees
also smoothness.}
Existence of such special solutions has been given for
specific models in the large-amplitude case --- for example,
the Broadwell model, \cite{Broa64} ---, or for general
relaxation system in the small--amplitude case, 
see \cite{YoZu00, MaZu02}.
By changing frame, such travelling wave can be assumed,
without loss of generality, as stationary solution of
(\ref{relax}), i.e. solution of the form 
 \begin{equation}\label{relaxshock}
  W=W(x),\qquad\qquad W(\pm\infty)=W_\pm.
 \end{equation} 
where $W_\pm=(u_\pm,v^*(u_\pm))$ with $u_\pm$ denoting the
state connected by the corresponding relaxed shock wave.

The next natural question to answer is whether such
steady states are stable or unstable.
The equation for the perturbation $w:=\tilde w-W$ is
\begin{equation*}
   w_t+\bigl(F(W+w)-F(W)\bigr)_x=Q(W+w)-Q(W).
 \end{equation*}
and the corresponding linearized equation is
 \begin{equation}\label{linearized}
   w_t={\mathcal L} w:=-\bigl(\diff F(W)w\bigr)' +\diff Q(W)w. 
  \end{equation}
Thus, the linearized eigenvalue equation is
  \begin{equation}\label{resolvent}
    (\lambda\,I-{\mathcal L})w
    =\lambda\,w+\bigl(\diff F(W)w\bigr)' -\diff Q(W)w=0. 
  \end{equation}
From now on, we consider equation (\ref{resolvent}) in 
the Sobolev space $H^1$ and we say that $\lambda\in\C$
is an eigenvalue of $\mathcal L$ if there exists a function
$w\in H^1\setminus\{0\}$, such that (\ref{resolvent}) 
holds.\footnote{On the region $\Real \lambda \ge 0$ and $\lambda\ne 0$ 
that we will consider, and under our hypotheses (A1)--(A2) below,
$H^1$ spectrum agrees with $L^p$ spectrum for any
$1\le p\le \infty$; see \cite{MaZu05}.}

By differentiating the equation satisfied by the profile $W$,
we get
  \begin{equation*}
    \bigl(\diff F(W)W'\bigr)' -\diff Q(W)W'=0. 
  \end{equation*}
In the noncharacteristic case, i.e. $\diff F(W_\pm)$ invertible,
$W'$ decays exponentially fast to zero as $|x|\to\infty$.
Thus, $\lambda=0$ is an eigenvalue of the linearized operator 
${\mathcal L}$. 
Hence, instability is related to the presence
of non-zero eigenvalues with non-negative real part.

\begin{definition}
The stationary solution $W$ is {\sf spectrally stable} in $H^1$,
if for any $\lambda\in\C\setminus\{0\}$,  $\Real\lambda\geq 0$,
whenever $w\in H^1$ solves the resolvent equation (\ref{resolvent}) 
then $w\equiv 0$.
\end{definition}  
 
Results in \cite{MaZu02, MaZu05} show that, under additional
assumptions, spectral stability implies both linear and nonlinear
orbital stability.
Hence, determining whether spectral stability holds or not
is the key issue for determining nonlinear stability of relaxation
shock profiles.
The aim of the present paper is to prove a general result on
spectral stability of relaxation shock profiles assuming smallness
of $\varepsilon:=|W_+-W_-|$.
\medskip

We divide the assumptions into three groups: {\sf A}, {\sf B} 
and {\sf C}, referring, respectively, to the relaxation system 
(\ref{relax}), to the the relaxed system (\ref{relaxed}), to the 
relaxation shock profile (\ref{relaxshock}).
These are imposed on a small neighborhood $\cW\subset \R^n$ about an
equilibrium point $w_0=(u_0,v_0)$, $q(w_0)=0$, or equivalently 
$v_0=v_*(u_0)$, and a neighborhood $\cU^*$ of $u_0$ such
that the graph of $v_*$ over $\cU^*$ is contained in $\cW$.
\medskip

{\sl
\noindent
{\sf A1.} There exists a smooth function $A_0=A_0(W)$ from 
$\cW\subset \R^n$ to the set of real symmetric positive definite
matrices such that $(A_0\diff F)(W)$ is symmetric
for any $W$ under consideration, and 
\begin{equation}\label{propA0}
 \Real\langle w,(A_0\,\diff Q)(W_\pm)\,w\rangle
 \leq -c|\Pi_\pm w|_{{}_{L^2}}^2,
\end{equation}
for some $c>0$, where $\Pi_\pm\,w:=\diff Q(W_\pm)\,w$.\\
{\sf A2.} {\it (Shizuta--Kawashima condition)} There exists a 
smooth skew-symmetric matrix--valued  function $K$, 
depending on $\diff F,\diff Q$ and $A_0$ such that 
\begin{equation}\label{skcondition}
   \Real\,(K \diff F-A_0\,\diff Q)>0.
\end{equation}
{\sf B1.} There exists a smooth function $a_0=a_0(u)$ from 
$\cU^*\subset \R^p$ to the set of real symmetric positive definite
matrices such that $a_0\,\diff f^*(u)$ is symmetric
for any $u$ under consideration and $a_0\,b_0$ 
is positive semidefinite, $\Real a_0b_0\ge 0$, where $b_0$ as
defined in \eqref{bs} is the associated Chapman--Enskog viscosity.\\
{\sf B2.} {\it (reduced Shizuta--Kawashima condition)} There exists a 
smooth skew-symmetric matrix--valued  function $k$, 
depending on $\diff f^*,b_0$ and $a_0$ such that 
\begin{equation}\label{rskcondition}
   \Real\,(k\,\diff f^*-a_0\,b_0)>0.
\end{equation}
{\sf B3.} 
{\it (Simplicity, genuine nonlinearity of principle equilibrium characteristic)}
There exists $c_0>0$ such that there is a single eigenvalue
$\alpha_0(u)$ of $\diff f^*(u)$ that has absolute 
value $<c_0$ on $\cU^*$,  with all others of absolute value $\ge 2c_0$.  
Moreover,
\begin{equation}\label{gnl}
\diff \alpha_0(u) \cdot r_0(u)=:\eta(u) \ne 0 
\end{equation}
on $\cU^*$, where $r_0(u)$ denotes
the unit right eigenvector associated with $\alpha_0(u)$. \\
{\sf C.}
There exists $C>0$ such that for any $x\in\R$ there hold
\begin{equation}\label{sizeder}
   |W'|_{{}_{L^\infty}}\leq C|W_+-W_-|^2,\qquad\qquad
   |W''(x)|\leq C|W_+-W_-|\,|W'(x)|,
\end{equation}
and
\begin{equation}\label{dirder}
   \Big|\frac{W'}{|W'|} +\sgn (\eta) r_0\Big|\leq C\,|W_+-W_-|^2.
\end{equation}
}

\begin{remark}\label{strucrmk}
The apparently restrictive  {\sf A1--A2}, {\sf B1--B2}
in fact all follow from the standard assumptions that
(i) there exist a positive definite symmetrizer $A_0$, $A_0dF$ symmetric,
that {\it at equilibrium points} simultaneously symmetrizes 
$dQ$, $A_0dQ$ symmetric (weak simultaneous symmetrizability), and
(ii) at equilibrium points, no eigenvector of $dF$ is in the kernel
of $dQ$ (genuine coupling);
see Lemma \ref{struclem}, Appendix \ref{strucapp}.
These two assumptions hold quite generally
in applications, in particular for discrete kinetic
equations and moment closure systems \cite{Yong01}.
Assumption {\sf B3} is standard and easily checked.
\end{remark}

\begin{remark}
Assumption {\sf C} is satisfied for a family of profiles near $w_0$ if:

 (i) $\diff F$ is invertible and $\alpha_0(u_0)=0$ (see Appendix of \cite{MaZu02}).

(ii) $\diff F$ is constant, $\alpha_0(u_0)=0$, and dimension $N$ bounded, 
e.g. in the case of discrete kinetic models with upper bound on the
number of modes (see \cite{MaZu05}).
It has been shown to hold also for the infinite-dimensional case
of the Boltzmann equation \cite{LiYu04}.\\
\end{remark}

With these assumptions, our main result is as follows.

\begin{theorem}[Spectral stability]\label{thm:spst}
Under assumptions 
{\sf A1--A2, B1--B2--B3}, and {\sf C,}
for $\varepsilon:=|W_+-W_-|$ sufficiently small,
the relaxation shock $W$ is spectrally stable.
\end{theorem}

As the argument is somewhat complicated, it may be helpful
to outline here the structure of the proof.
We start by carrying out the following by-now-standard
``Kawashima-type'' energy estimates on the relaxation system.

\begin{proposition}\label{prop:kawa} 
Assume hypothesis {\sf A1-A2} and {\sf C}. 
Let $\lambda\in\C$ such that $\Real\lambda\geq 0$
and let $w$ be a solution of (\ref{resolvent}).
Then for $\varepsilon$ sufficiently small,  there hold:
 \begin{eqnarray}
  \label{kawa}
   &\Real\,\lambda\, |w|^2_{{}_{L^2}}+|\Pi\,w|^2_{{}_{L^2}}
   +|w'|^2_{{}_{L^2}}\leq C\varepsilon^2\,|u|^2_{{}_{L^2}}
    \qquad &\textrm{\it (Kawashima estimate)}\\
  \label{evalueEst}
 &\Real\lambda \leq C\,\varepsilon^2,\qquad 
   |\Imag\lambda|\leq C\,\varepsilon. &
  \end{eqnarray} 
where $\Pi:=\Pi_++\Pi_-$.
\end{proposition}  

Evidently, it remains only to obtain estimates on the equilibrium
variable $|u|_{{}_{L^2}}^2$.
To this end, we carry out an approximate Chapman--Enskog expansion,
keeping track of error terms, to obtain an effective viscous
system for $u$ of the same symmetric dissipative type, but with
error terms coming from higher derivatives.
Applying Goodman-type energy estimates to the integrated version of this
reduced system, following \cite{HuZu02},
we obtain the desired bounds on $|u|_{{}_{L^2}}^2$ modulo errors
consisting of higher-derivative and dissipative  terms (denoted by
$\hat v$ in Section \ref{sec:redsystem}).
Observing that these, by \eqref{kawa}, may be absorbed
in lower-order and equilibrium terms, we are done.

More precisely, we establish the following bounds on 
the reduced system.
Here we use the following notation for the 
$W'$--weighted $L^2$--norm
\begin{equation*}
 |z|_{{}_{W'}}:=|\sqrt{|W'|}\,z|_{{}_{L^2}}
  =\left(\int_{\R} |z|^2\,|W'|(x)\,dx\right)^{1/2}
\end{equation*}
The space of functions with  bounded $|\cdot|_{{}_{W'}}$ will
be denoted by $L^2_{{}_{W'}}$.
Since $W'$ is bounded,
there holds $|z|_{{}_{W'}}\leq C\,|z|_{{}_{L^2}}$ for some $C>0$
(a natural choice is $C:=|W'|_{{}_{L^\infty}}^{1/2}$).
Hence $L^2\subset L^2_{{}_{W'}}$ with continuous injection.
The opposite inequality is false since  
$W'$ decays (exponentially fast) to zero as $|x|\to\infty$.
  
\begin{proposition}\label{prop:kawareduced}
Assume hypothesis {\sf B1-B2} and {\sf C}.
Let $z$ be defined as follows
\begin{equation*}
 z(x):=\int_{-\infty}^x u(y)\, dy.
\end{equation*}
Then, for $\varepsilon$ sufficiently small, there holds
\begin{equation}\label{kawared}
  \Real\lambda\,|z|_{{}_{L^2}}^2+|u|_{{}_{L^2}}^2\leq C|z|_{{}_{W'}}^2.
 \end{equation}
\end{proposition}  
  
Thanks to (\ref{kawared}), it is sufficient to control
the weighted norm $|z|_{{}_{W'}}$.
  
\begin{proposition}\label{goodprop}
Under assumptions  {\sf B1--B2--B3} and {\sf C},
for $\varepsilon$  sufficiently small, there holds:
\begin{equation}\label{good}
   \Real\lambda|z|_{{}_{L^2}}^2+|z|_{{}_{W'}}^2
      \leq C\bigl(\varepsilon|u|_{{}_{L^2}}^2+\varepsilon|\Pi w|_{{}_{L^2}}^2
   +\varepsilon^{-1}|w'|_{{}_{L^2}}^2\bigr)
    \quad \textrm{\it (Goodman estimate)}
  \end{equation}
where $\Pi:=\Pi_++\Pi_-$.
\end{proposition}  
 
Theorem \ref{thm:spst} is an immediate consequence of  
Propositions \ref{prop:kawa}, \ref{prop:kawareduced}
and \ref{goodprop}.
 
\begin{proof}[Proof of Theorem \ref{thm:spst}.]
Combining \eqref{kawa} and (\ref{good}), we obtain
\begin{equation*}
   |z|_{{}_{W'}}^2 \leq C\,\varepsilon\,|u|_{{}_{L^2}}^2.
\end{equation*}
With (\ref{kawared}), this gives
 \begin{equation*}
    |u|^2_{{}_{L^2}}  \leq C\,|z|_{{}_{W'}}^2
    \leq C\,\varepsilon^2\,|u|^2_{{}_{L^2}},
\end{equation*}
showing that, 
if $\varepsilon=|W_+-W_-|$ is small enough, $w\equiv 0$.
\end{proof}

\subsection*{Discussion and open problems}

We remark briefly on the setting of these results.  
Small-amplitude existence
and stability were shown in \cite{YoZu00,MaZu02} and \cite{PlZu04}
under the additional noncharacteristicity assumption 
$\det \diff F\ne0$,
or, equivalently, the condition that characteristic speeds of the
background system do not vanish relative to the shock speed.
This hypothesis suffices to treat simple model problems such
as the Broadwell or Jin--Xin equations.  However, as discussed
in \cite{MaZu05}, it is unrealistic
for models derived by discretization or moment closure from 
kinetic equations, since these may possess characteristics of
any speed.
Thus, it is highly desirable to remove this technical hypothesis,
as we do here.
The combination of Goodman- and Kawashima-type energy estimates was
used in \cite{HuZu02} to treat stability of viscous
shock profiles for systems with real viscosity.
A similar, but more complicated argument combining these ingredients
was used in \cite{LiYu04} to treat stability of Boltzmann profiles. 
These results motivate the present analysis, which 
essentially interpolates between the two. 

Interesting open problems are verification of linearized and
nonlinear stability in the same setting, assuming spectral stability, 
and the direct verification of {\sf C} using stability estimates
together with known bounds on the profile for the reduced system
following the philosophy set out in \cite{LiYu04}.

\subsection*{Notations and (very) basic tools}

Given $w_1, w_2\,:\,\R\to\C^n$, we denote by $\langle w_1, w_2\rangle$ the 
scalar product defined as follows
\begin{equation*}
 \langle w_1, w_2\rangle:=\int_{\R} \overline{w_1(x)}\cdot w_2(x)\,dx
\end{equation*}
where $\bar w$ denotes the complex conjugate vector of $w$.
Given $A$, $n\times n$ matrix with complex entries, there
holds
\begin{equation*}
 \Real \langle w, A w\rangle=\frac{1}{2}\left(\langle w, A w\rangle
 +\overline{\langle w, A w}\rangle\right)
 =\langle w, A^* w\rangle
\end{equation*}
where $A^*:=(A+\overline{A^t})/2$.
 
If $S\,:\,\R\to\R^{n\times n}$ is such that $S(x)$ is symmetric for any $x$, then
\begin{equation*}
 \langle w, S\,w'\rangle=\int_{\R} \overline{w}\cdot Sw'\,dx
 =-\int_{\R} \overline{Sw'}\cdot w\,dx-\int_{\R} \overline{w}\cdot S'\,w\,dx
 =-\overline{\langle w, S\,w'\rangle}-\langle w, S'\,w\rangle
\end{equation*}
Hence
\begin{equation}\label{symm}
 \Real\langle w, S\,w'\rangle=-\frac12\langle w, S'\,w\rangle.
\end{equation}
Similarly, if $K$ is skew--symmetric, then
\begin{equation}\label{skewsymm}
\Imag\,\langle w',Kw\rangle=-\frac12 \langle w,K'\,w\rangle
\end{equation}
In particular, if $K$ is constant, $\langle w',Kw\rangle$ is a 
real number.
\medskip

From here on, we will denote with $O(1)$ any 
function of $x, W$ and $\lambda$, locally 
bounded  in $\{(x, W,\lambda)\,:\,\Real\lambda\geq 0\}$.
As a consequence, given the functions $f,g\in L^2$
and $h\in L^2_{{}_{W'}}$, the following estimates
hold
 \begin{equation}\label{chromo}
 |\langle f,O(1)\,W'g\rangle|
 \leq C|W'|_{{}_{L^\infty}}\left( \eta|f|_{{}_{L^2}}^2
  +\eta^{-1}|g|_{{}_{L^2}}^2\right)
 \end{equation}
 \begin{equation}\label{doris}
 |\langle h,O(1)\,W'g\rangle|
 \leq C\left(\eta|h|_{{}_{W'}}^2
  +\eta^{-1}|W'|_{{}_{L^\infty}}|g|_{{}_{L^2}}^2\right)
 \end{equation}
where $\eta$ is any strictly positive constant
and $C$ is a constant independent on $\eta$.
 
\section{Estimates for the full system}

\begin{lemma}\label{lemma:fried}
Let $\varepsilon:=|W_+-W_-|$ and
assume hypothesis {\sf A1} and {\sf C}. 
Let $\lambda\in\C$ such that $\Real\lambda\geq 0$
let $w$ be a solution of (\ref{resolvent})
and let $K$ be any constant skew-symmetric matrix.
Then for $\varepsilon, \eta>0$ both sufficiently small, there hold
\begin{eqnarray}
 \label{fried0eps}
    & \Real\,\lambda\, |w|^2_{{}_{L^2}}+|\Pi\,w|^2_{{}_{L^2}}
   \leq  C\,\varepsilon^2\,|w|^2_{{}_{L^2}},\\
 \label{fried1eps}  
   &  \Real\lambda |w'|^2_{{}_{L^2}}
    -\Real\langle w',A_0\,\diff Q\,w'\rangle
    \leq C\,\varepsilon^2\,|w|_{{}_{H^1}}^2;\\
 \label{interKawa}
   &  \Real\,\langle w',K\,\diff F\,w'\rangle 
    \leq  C\,\bigl(\varepsilon^2|w|^2_{{}_{L^2}}
   +\eta^{-1}\,|\Pi w|^2_{{}_{L^2}}
   +(\varepsilon^2+\eta)|w'|^2_{{}_{L^2}}\bigr)
\end{eqnarray}
where $\Pi:=\Pi_++\Pi_-$ and $C$ denotes a constant 
independent on $\varepsilon$ and $\eta$.
\end{lemma}

\begin{proof}
Taking the scalar product of  $A_0(W)\,w$ against (\ref{resolvent}), 
we obtain
 \begin{equation}\label{A0weq}
   \lambda\langle A_0\,w,w\rangle+\langle A_0\,w,\bigl(\diff F\,w\bigr)'\rangle
    -\langle A_0\,w,\diff Q\,w\rangle =0.
 \end{equation}
Hence, using (\ref{symm}), we get
 \begin{equation*}
    \Real\lambda\langle A_0\,w,w\rangle
   -\Real\langle w,A_0\,\diff Q\,w\rangle\leq
   -\Real\langle A_0\,w, \diff^2 F\,W'\,w\rangle
   +\frac12\Real\langle w,\diff(A_0\,\diff F)\,W'\,w\rangle.
 \end{equation*}
Since $A_0$ is positive definite, there holds for some $C>0$
 \begin{equation}\label{discodoris}
  \Real\lambda\, |w|^2_{{}_{L^2}}-\Real\langle w,A_0\,\diff Q\,w\rangle
  \leq C\,|W'|_{{}_{L^\infty}}|w|^2_{{}_{L^2}}.
 \end{equation}
 Let us set 
\begin{equation*}
 \Phi(W):=\frac{|W-W_+|}{|W_+-W_-|}(A_0\,\diff Q)(W_-)
  +\frac{|W-W_-|}{|W_+-W_-|}(A_0\,\diff Q)(W_+).
\end{equation*}
Then there holds, for some $C>0$, 
\begin{equation*}
 |\Phi(W)-(A_0\,\diff Q)(W)|\leq C|W-W_-||W-W_+|.
\end{equation*}
Therefore
\begin{equation*}
 \Real \langle w,A_0\,\diff Q\,w\rangle
 \leq \Real\langle w, \Phi(W)\,w\rangle
 +C|W-W_-|_{{}_{L^\infty}}|W-W_+|_{{}_{L^\infty}}|w|_{{}_{L^2}}^2
\end{equation*}
Thanks to (\ref{propA0}), we get 
\begin{equation*}
 \Real \langle w,A_0\,\diff Q\,w\rangle
 \leq -c|\Pi w|_{{}_{L^2}}^2
 +C|W-W_-|_{{}_{L^\infty}}|W-W_+|_{{}_{L^\infty}}|w|_{{}_{L^2}}^2
\end{equation*}
for some $C,c>0$.
Hence, using (\ref{discodoris}), we obtain
 \begin{equation}
   \textrm{Re}\,\lambda\, |w|^2_{{}_{L^2}}+|\Pi\,w|^2_{{}_{L^2}}
   \leq C\,\bigl(
   |W-W_+|_{{}_{L^\infty}}|W-W_-|_{{}_{L^\infty}}
   +|W'|_{{}_{L^\infty}}\bigr)|w|^2_{{}_{L^2}}.
 \label{fried0}
 \end{equation}
In term of $\varepsilon$, we get the 
{\it 0-th order Friedrichs estimate} (\ref{fried0eps}).

Differentiating (\ref{resolvent}) with respect to $x$,
\begin{equation}
    \lambda\,w'+\bigl(\diff F(W) w\bigr)'' -(\diff Q(W) w)'=0. 
  \label{diffresolvent}
  \end{equation}
Taking the scalar product of $A_0(W)\,w'$ 
against (\ref{diffresolvent}), we get
\begin{equation}\label{atromaculata}
   \lambda\langle A_0\,w',w'\rangle-\langle w',A_0\,\diff Q\,w'\rangle=
  -\langle A_0\,w',\bigl(\diff F\,w\bigr)''\rangle
  +\langle A_0\,w',\diff^2Q\,W'\,w\rangle.
\end{equation}
Since
\begin{equation*}
 \langle A_0\,w',\bigl(\diff F\,w\bigr)''\rangle
 = \langle A_0\,w',\diff^3F\,W'\,W'\,w\rangle
 + \langle A_0\,w',\diff^2F\,W''\,w\rangle
 \end{equation*}
\begin{equation*}
 \qquad\qquad\qquad 
 +2\langle A_0\,w',\diff^2F\,W'\,w'\rangle
  +\langle w',A_0\,\diff F\,w''\rangle,
 \end{equation*}
taking the real part and using (\ref{symm}), we obtain 
\begin{equation*}
\Real\langle A_0\,w',\bigl(\diff F\,w\bigr)''\rangle
  = \langle w',O(1)W'\,w\rangle
  + \langle w',O(1)W''\,w\rangle
 \end{equation*}
\begin{equation*}
 \qquad\qquad\qquad 
    +\langle w',O(1)W'\,w'\rangle
   -\frac12\langle w',\diff(A_0\,\diff F)\,W'\,w'\rangle.
\end{equation*}
Hence, the following estimates holds
\begin{equation*}
   \bigl|\Real\langle A_0\,w',\bigl(\diff F\,w\bigr)''\rangle\bigr|
    \leq C\left(|W'|_{{}_{L^\infty}}+|W''|_{{}_{L^\infty}}\right)|w|_{{}_{H^1}}^2
\end{equation*}
\begin{equation*}
   \bigl|\Real\langle A_0\,w',\diff^2Q\,W'\,w\rangle\bigr|
    \leq C\,|W'|_{{}_{L^\infty}}|w|_{{}_{H^1}}^2.
\end{equation*}
Therefore, from (\ref{atromaculata}), using (\ref{sizeder}), we deduce
\begin{equation*}
   \Real\lambda |w'|^2_{{}_{L^2}}-\Real\langle w',A_0\,\diff Q\,w'\rangle
   \leq C\,|W'|_{{}_{L^\infty}}\,|w|_{{}_{H^1}}^2.
\end{equation*}
 Thus in term of $\varepsilon$, we obtain the 
 {\it 1-st order Friedrichs estimate} (\ref{fried1eps}).
 
Now, let $K$ be any constant skew-symmetric matrix.
Applying $K$ to the resolvent equation (\ref{resolvent}) and 
multiplying by $w'$, we get
\begin{equation*}
   \lambda\,\langle w',Kw\rangle+\langle w',K\bigl(\diff F\,w\bigr)'\rangle
    -\langle w',K\,\diff Q\,w\rangle=0. 
\end{equation*}
Taking the real parts and rearranging the terms, we obtain
\begin{equation*}
   \Real\,\langle w',K\,\diff F\,w'\rangle=
    -\Real\bigl(\lambda\,\langle w',Kw\rangle\bigr)
    -\Real\,\langle w',K\,\diff^2 F\,W'\,w\rangle 
    +\Real\,\langle w',K\,\diff Q\,w\rangle.
\end{equation*}
Hence, thanks to (\ref{skewsymm}), there holds
$\textrm{Im}\,\langle w',Kw\rangle=0$, since $K$ is constant.
Therefore, for $\Real\lambda\geq 0$, we obtain 
\begin{equation*}
   |\Real\bigl(\lambda\,\langle w',Kw\rangle\bigr)|
   = \Real\lambda\,|\Real\, \langle w',Kw\rangle|
    \leq C\,\Real\lambda\,|w|^2_{{}_{H^1}}.
\end{equation*}
Let us set 
\begin{equation*}
 \Psi(W):=\frac{|W-W_+|}{|W_+-W_-|}K\,\diff Q(W_-)
  +\frac{|W-W_-|}{|W_+-W_-|}K\,\diff Q(W_+).
\end{equation*}
Then there holds, for some $C>0$, 
\begin{equation*}
 |\Psi(W)-K\,\diff Q(W)|\leq C|W-W_-||W-W_+|.
\end{equation*}
Therefore
\begin{equation*}
 \Real \langle w',K\,\diff Q\,w\rangle
 \leq \Real\langle w', \Psi(W)\,w\rangle
 +C|W-W_-|_{{}_{L^\infty}}|W-W_+|_{{}_{L^\infty}}|w|_{{}_{H^1}}^2
\end{equation*}
Thanks to (\ref{propA0}), we get 
\begin{equation*}
 \Real \langle w',K\,\diff Q\,w\rangle
 \leq C\,\bigl(\eta^{-1}\,|\Pi w|^2_{{}_{L^2}}
   +\eta |w'|^2_{{}_{L^2}}\bigr)
 +C|W-W_-|_{{}_{L^\infty}}|W-W_+|_{{}_{L^\infty}}|w|_{{}_{L^2}}^2
\end{equation*}
where $\eta>0$ is a positive constant to be chosen later on
(small enough).
Hence, the following estimate holds
\begin{equation*}
 \begin{aligned}
   \Real\,\langle w',K\,\diff F\,w'\rangle=
    C\,\bigl\{(\Real\lambda+|W'|_{{}_{L^\infty}}
     +|W-W_-|_{{}_{L^\infty}}&|W-W_+|_{{}_{L^\infty}})|w|^2_{{}_{H^1}}\\
    &+\eta^{-1}\,|\Pi w|^2_{{}_{L^2}}
   +\eta |w'|^2_{{}_{L^2}}\bigr\}
\end{aligned}   
\end{equation*}
By (\ref{fried0eps}), $\textrm{Re}\lambda\leq C\, \varepsilon^2$,
hence, in term of $\varepsilon$, we get (\ref{interKawa}).
This concludes the proof of Lemma \ref{lemma:fried}.
\end{proof}

Assuming, in addition, hypothesis {\sf A2}, we 
prove estimates (\ref{kawa}) and (\ref{evalueEst}).

\begin{proof}[Proof of Proposition \ref{prop:kawa}]
Thanks to the small amplitude assumption, it is possible
to choose $K=K(W_+)$ constant in the Shizuta--Kawashima condition
\eqref{skcondition}, since this is an open condition so persists
under small perturbations.
Hence, summing estimates (\ref{fried0eps})-(\ref{fried1eps}) 
with (\ref{interKawa}),  we obtain
\begin{equation*}
   \Real\,\lambda\, |w|^2_{{}_{L^2}}+|\Pi\,w|^2_{{}_{L^2}}
   +|w'|^2_{{}_{L^2}}\leq C\,\varepsilon^2\,|w|_{{}_{H^1}}^2
   +C\,(1+\eta^{-1})\varepsilon^2|w|^2_{{}_{L^2}}
   +C(\varepsilon^2+\eta)|w'|^2_{{}_{L^2}}
   \end{equation*}
 which yields 
\begin{equation*}
   \Real\,\lambda\, |w|^2_{{}_{L^2}}+|\Pi\,w|^2_{{}_{L^2}}
   +|w'|^2_{{}_{L^2}}
   \leq C\,\varepsilon^2\,|w|_{{}_{L^2}}^2
   \end{equation*}
for $\varepsilon$ and $\eta$ sufficiently small.
Since $|w|_{{}_{L^2}}^2\leq C\bigl(|u|_{{}_{L^2}}^2+|\Pi\, w|_{{}_{L^2}}^2\bigr)$ 
we get the estimate \eqref{kawa} for $\varepsilon$ small.

Estimate  (\ref{fried0eps}) implies the bound on the real
part of the eigenvalue $\lambda$.
Taking the imaginary part of (\ref{A0weq}), 
\begin{equation*}
 \Imag\bigl(\lambda\,\langle A_0\,w,w\rangle\bigr)
  =-\Imag\langle A_0\,w,\diff^2F\,W'\,w\rangle
 -\Imag\langle A_0\,w,\diff F\,w'\rangle+\Imag\langle A_0\,w,\diff Q\,w\rangle.
\end{equation*}
Hence, for $\eta>0$ to be chosen, 
\begin{equation*}
 |\Imag\lambda|\,|w|_{{}_{L^2}}^2
  \leq C\left(|W'|_{{}_{L^\infty}}\,|w|_{{}_{L^2}}^2
 +\eta\,|w|_{{}_{L^2}}^2+\eta^{-1}|w'|_{{}_{L^2}}^2
 +\eta^{-1}|\Pi w|_{{}_{L^2}}^2\right).
\end{equation*}
Thanks to (\ref{sizeder}) and (\ref{kawa}), we get 
\begin{equation*}
 |\Imag\lambda|\,|w|_{{}_{L^2}}^2\leq C\left(\varepsilon^2+\eta
 +\eta^{-1}\varepsilon^2\right)\,|w|_{{}_{L^2}}^2.
\end{equation*}
Thus, choosing $\eta=\varepsilon$, we obtain 
the result for $\varepsilon$ small enough.
\end{proof}
 
\section{The reduced system for the conserved 
integrated variables}\label{sec:redsystem}

As stressed in the Introduction, the next step consists in estimating the
conserved densities $u$ in term of an appropriate weighted $L^2-$norm
of the conserved quantities $z$, defined by
\begin{equation}\label{zdef}
 z(x):=\int_{-\infty}^x u(y)\, dy.
\end{equation} 
The first step is to the deduce a balance law satisfied by the variable $z$
with source terms depending on $\Pi w$ and $w'$.

Setting 
\begin{equation*}
 \diff F:=\left(\begin{array}{cc}
  A_{11} & A_{12}\\  A_{21} & A_{22}
 \end{array}\right)\quad\textrm{and}\quad
 \diff Q:=\left(\begin{array}{cc}
  0 & 0\\  q_{1} & q_{2}
 \end{array}\right),
\end{equation*}
equation (\ref{resolvent}) can be rewritten as
\begin{equation*}
   \left\{\begin{aligned}
    &\lambda\,u+(A_{11}\,u+A_{12}\,v)'=0,\\
    &\lambda\,v+(A_{21}\,u+A_{22}\,v)'-q_1 u-q_2 v=0.
   \end{aligned}\right.
\end{equation*} 
Let $\hat v:=  q_2^{-1}\,q_1\,u+v$.  
In particular, 
$c_1|\hat v|_{{}_{L^2}}\leq |\Pi\,w|_{{}_{L^2}}\leq 
c_2|\hat v|_{{}_{L^2}}$ for some $c_1,c_2>0$.  
Hence, in the following, we consider the variable $\hat v$ in
place of $\Pi w$.

Then the couple $(u,\hat v)$ satisfies
\begin{equation*}
    \left\{\begin{aligned}
    &\lambda\,u+\bigl(a\,u+A_{12}\,\hat v\bigr)'=0,\\
    &(\lambda I_r-q_2)\,\hat v
     +\bigl(c\,u+A_{22}\,\hat v\bigr)'
     +q^{-1}_2 q_1\bigl(a\,u+A_{12}\,\hat v\bigr)'=0
    \end{aligned}\right.
\end{equation*}
where 
\begin{equation*}
    a:=A_{11}-A_{12} q^{-1}_2 q_1,\qquad
    c:= A_{21}-A_{22} q^{-1}_2 q_1.
\end{equation*}
With $z$ defined in (\ref{zdef}), we can write the above system as 
\begin{equation*}
    \left\{\begin{aligned}
    &\lambda\,z+a\,z'+A_{12}\,\hat v=0,\\
    &(\lambda I_r-q_2)\,\hat v
     +\bigl(c\,z'+A_{22}\,\hat v\bigr)'
     +q^{-1}_2 q_1\bigl(a\,z'+A_{12}\,\hat v\bigr)'=0
    \end{aligned}\right.
\end{equation*}
Next the idea is to obtain an expression for $\hat v$ from the
second equation and inserting it in the first one, in order to obtained
a reduced system of viscous conservation laws with source terms.
Since we want to derive energy estimates, it is useful to change variables
in the first equation in order to symmetrize the term containing the first order
derivative.

Let $a_0$ be a symmetric and positive definite matrix such that $a_0a$ is
symmetric, as in assumption {\sf B1}.
Let $\tilde z:=a_0^{1/2} z$. 
The new variable $\tilde z$ and the variable $\hat v$ satisfy
\begin{equation*}
    \left\{\begin{aligned}
    &\lambda\,\tilde z+a_0^{1/2}\,a\,(a_0^{-1/2}\,\tilde z)'+a_0^{1/2}\,A_{12}\,\hat v=0,\\
    &(\lambda I_r-q_2)\,\hat v
     +\bigl(c\,(a_0^{-1/2}\,\tilde z)'+A_{22}\,\hat v\bigr)'
     +q^{-1}_2 q_1\bigl(a\,(a_0^{-1/2}\,\tilde z)'+A_{12}\,\hat v\bigr)'=0
    \end{aligned}\right.
\end{equation*}
Hence
\begin{equation*}
    \left\{\begin{array}{l}
    \lambda\,\tilde z+\tilde a\,\tilde z'+a_0^{1/2}\,A_{12}\,\hat v
      =O(1)\,W'\,\tilde z',\\
    (\lambda I_r-q_2)\,\hat v
     +\bigl(c\,a_0^{-1/2}\,\tilde z'+O(1)W'\tilde z+O(1)\hat v\bigr)'\\
     \qquad\qquad\quad
     +q^{-1}_2 q_1\bigl(a\,a_0^{-1/2}\,\tilde z'+O(1)W'\tilde z
      +O(1)\hat v\bigr)'=0
    \end{array}\right.
\end{equation*}
where the matrix $\tilde a:=a_0^{1/2}\,a\,a_0^{-1/2}$ is symmetric.
From the second equation, using (\ref{sizeder}), we get
\begin{equation}\label{tildev}
    \hat v=-(\lambda I_r-q_2)^{-1}\bigl(c+q^{-1}_2 q_1a\bigr)a_0^{-1/2}\,\tilde z''
       +O(1)W'\bigl(\tilde z'+\varepsilon\tilde z+\hat v\bigr)+O(1)\hat v',
\end{equation}
or, equivalently, using the $O(1)$ notation, 
\begin{equation}\label{shortHatv}
   \hat v=O(1)W'\bigl(\tilde z'+\varepsilon\tilde z+\hat v\bigr)+O(1)w'.
 \end{equation}
Plugging (\ref{tildev}) in the equation satisfied by $\tilde z$, we get
\begin{equation}\label{notenough}
 \lambda\,\tilde z+\tilde a\,\tilde z'-\tilde b\tilde z''
      =O(1)W'\bigl(\tilde z'+\varepsilon\tilde z+\hat v\bigr)+O(1)\hat v'
 \end{equation}
where
\begin{equation*}
 b:=A_{12}(\lambda I_r-q_2)^{-1}\bigl(c+q^{-1}_2 q_1a\bigr)
 \qquad\textrm{and}\qquad
 \tilde b:=a_0^{1/2}\,b\,a_0^{-1/2}.
\end{equation*} 
Estimate on $\tilde z$ will be obtained by mutliplying (\ref{notenough})
by $z$ and integrating. 
With the present form, we would obtain a "bad" term 
$\langle \tilde z, O(1)\tilde v'\rangle$.
For this reason, it is useful to use the relation (\ref{shortHatv}) to obtain
the following new version of (\ref{notenough})
\begin{equation*}
 \lambda\,\tilde z+\tilde a\,\tilde z'-\tilde b\tilde z''
      =O(1)W'\bigl(\tilde z'+\varepsilon\tilde z+\hat v\bigr)
      +O(1)\bigl(O(1)W'\bigl(\tilde z'+\varepsilon\tilde z+\hat v\bigr)+O(1)w'\bigr)'.
 \end{equation*}
For $\lambda=0$, the term $b$ represents the viscosity term 
given by the Chapman--Enskog expansion for the relaxation system.
Hence, it is significant to decompose $b$ as follows
\begin{equation*}
 b=b_0+\lambda b_1
\end{equation*}
where matrices $b_0$ and $b_1$ are given by
\begin{equation*}\label{bs}
   b_0:=-A_{12}q_2^{-1}(c+q^{-1}_2 q_1a),\quad
   b_1:=A_{12}(\lambda-q_2)^{-1}q_2^{-1}(c+q^{-1}_2 q_1a).
\end{equation*}
Hence we get the following equation 
satisfied by the variable $\tilde z$
\begin{equation*}
  \lambda\,\tilde z+\tilde a\,\tilde z'-\tilde b_0\,\tilde z''
   =\lambda \tilde b_1 \tilde z''+\Theta.
\end{equation*} 
where $\tilde b_i:=a_0^{1/2}\,b_i\,a_0^{-1/2}$ for $i=0,1$
and $\Theta$ is appropriately defined.
By assumptions {\sf B1.} on the reduced system, the matrix 
$\tilde b_0$ is symmetric and positive semidefinite.

From now on, we drop the tildas for shortness and consider the 
following equation
\begin{equation}\label{zequation} 
  \lambda\,z+a\,z'-b_0\,z''=\lambda b_1 z''+\Theta_1+\Theta_2'.
\end{equation} 
where $a$ is a symmetric matrix, $b_0$ is a symmetric,
positive semidefinite matrix, and
\begin{equation}\label{Thetaform}
 \left\{\begin{aligned}
 &\Theta_1:=O(1)W'\bigl(z'+\varepsilon z+\hat v+w'\bigr)\\
 &\Theta_2:=O(1)W'\bigl(z'+\varepsilon z+\hat v\bigr)+O(1)w'
  \end{aligned}\right.
\end{equation}

\begin{lemma}\label{lemma:friedRed}
Assume hypothesis {\sf C}.
Let $\lambda\in\C$ satisfying (\ref{evalueEst}),
let $z$ be a solution (\ref{zequation})
with $a$ symmetric, $b_0$ symmetric, positive semidefinite and
$\Theta_1, \Theta_2$ given in (\ref{Thetaform}) and let $k$ be a 
smooth function from $\R^n$ to the set of real skew-symmetric matrices.

Then, for $\varepsilon, \eta>0$ sufficiently small,
the following estimates hold:
\begin{eqnarray*}
  &\Real\lambda\,|z|_{{}_{L^2}}^2+\langle z',b_0\,z'\rangle
  \leq C\bigl(|z|_{{}_{W'}}^2+(\varepsilon+\eta)|z'|_{{}_{L^2}}^2
    +\varepsilon^2|\hat v|_{{}_{L^2}}^2
     +\eta^{-1}|w'|_{{}_{L^2}}^2\bigr)\\
   &\Real\langle z',k\,a\,z'\rangle
   \leq C\bigl(\eta|\Real\lambda||z|_{{}_{L^2}}^2
   +\varepsilon|z|_{{}_{W'}}^2
    +(\eta^{-1}\varepsilon^2+\eta)|z'|_{{}_{L^2}}^2
     +\varepsilon^2|\hat v|_{{}_{L^2}}^2
    +\eta^{-1}\,|w'|_{{}_{L^2}}^2\bigr)
\end{eqnarray*}
\end{lemma}

\begin{proof}
Taking the real part of the scalar product of $z$ against 
(\ref{zequation}), we get
\begin{equation*}
  \Real\lambda\,|z|_{{}_{L^2}}^2-\frac12 \langle z,\diff a\,W'z\rangle
  -\Real \langle z,b_0\,z''\rangle
  =\Real\langle z,\lambda\,b_1\,z''\rangle
  +\Real\langle z,\Theta\rangle.
\end{equation*}
having used the symmetry of $a$.
Since
\begin{equation*}
 \begin{aligned}
  \Real \langle z,b_0\,z''\rangle
  &=-\Real\langle z',b_0\,z'\rangle
  -\Real \langle z,\diff b_0\,W'z'\rangle\\
  &\leq -\Real\langle z',b_0\,z'\rangle
  +C(|z|_{{}_{W'}}^2+|W'|_{{}_{L^\infty}}|z'|_{{}_{L^2}}^2).
  \end{aligned}
\end{equation*}
we obtain, thanks to (\ref{sizeder}),
\begin{equation*}
  \Real\lambda\,|z|_{{}_{L^2}}^2+\Real\langle z',b_0\,z'\rangle
  \leq C\bigl(|z|_{{}_{W'}}^2+\varepsilon^2|z'|_{{}_{L^2}}^2
  +\Real\langle z,\lambda\,b_1\,z''\rangle
   +\Real\langle z,\Theta\rangle\bigr).
\end{equation*}
The term containing $b_1$ can be easily estimated by
\begin{equation*}
 \begin{aligned}
  \left|\Real \langle z,\lambda b_1\,z''\rangle\right|
  &\leq |\lambda|\left(|\langle z',b_1\,z'\rangle|
  +|\langle z,\diff b_0\,W'z'\rangle|\right)\\
  &\leq C|\lambda|\left(|z'|_{{}_{L^2}}^2
   +|z|_{{}_{W'}}^2+|W'|_{{}_{L^\infty}}|z'|_{{}_{L^2}}^2\right).
  \end{aligned}
\end{equation*}
Taking in account (\ref{sizeder}) and (\ref{evalueEst}), we get
\begin{equation*}
  \left|\Real \langle z,\lambda b_1\,z''\rangle\right|
  \leq C\,\varepsilon\left(|z'|_{{}_{L^2}}^2+|z|_{{}_{W'}}^2\right).
\end{equation*}
Therefore, we obtain
\begin{equation}\label{almostthere}
  \Real\lambda\,|z|_{{}_{L^2}}^2+\Real\langle z',b_0\,z'\rangle
  \leq C\bigl(|z|_{{}_{W'}}^2+\varepsilon|z'|_{{}_{L^2}}^2
   +\Real\langle z,\Theta\rangle\bigr).
\end{equation}
It remains to deal with the term with $\Theta$.
For what concerns $\Theta_1$, using (\ref{doris}) (with $\eta=1$) and 
(\ref{sizeder}), we have
\begin{equation*}
 |\Real\langle z,\Theta_1\rangle|\leq C|z|_{{}_{W'}}^2
  +C\varepsilon^2\bigl(|z'|_{{}_{L^2}}^2
  +|\hat v|_{{}_{L^2}}^2+|w'|_{{}_{L^2}}^2\bigr)
  \end{equation*}
The term with $\Theta_2$ can be dealt with integrating by parts
\begin{equation*}
 |\Real\langle z,\Theta_2'\rangle| =|\Real\langle z',\Theta_2\rangle|
 \leq  C\varepsilon|z|_{{}_{W'}}^2+C(\varepsilon^2+\eta)|z'|_{{}_{L^2}}^2
  +C\varepsilon^2|\hat v|_{{}_{L^2}}^2+C\eta^{-1}|w'|_{{}_{L^2}}^2
 \end{equation*}
where $\eta$ is any positive constant and $C$ is independent 
on $\eta$.
Inserting the last three estimates in (\ref{almostthere}), 
we get {\sf i.} in  Lemma \ref{lemma:friedRed}.
 
Applying $k$ to (\ref{zequation}),  $k$ as defined in {\sf B2},
and taking the $L^2$ scalar  product against $z'$, we get
\begin{equation}
   \Real\langle z',k\,a\,z'\rangle=-\Real\lambda\,\langle z', kz\rangle
   +\Real\langle z',k\,b\,z''\rangle+\Real\langle z',k\Theta\rangle.
\label{kawared1}
\end{equation}
Using the eigenvalue estimate on $\Real \lambda$, stated in 
Proposition \ref{prop:kawa}, we obtain
\begin{equation*}
   \Real\langle z',k\,a\,z'\rangle
   \leq C\left(\eta|\Real\lambda||z|_{{}_{L^2}}^2
   +(\eta^{-1}\varepsilon^2+\eta)\,|z'|_{{}_{L^2}}^2
    +\eta^{-1}\,|z''|_{{}_{L^2}}^2\right)+\Real\langle z',k\Theta\rangle
\end{equation*}
Finally, using once more (\ref{sizeder}),
for the term with $\Theta$ there hold 
\begin{equation*}
 |\Real\langle z',k\Theta_1\rangle|
 \leq C\varepsilon|z|_{{}_{W'}}^2+C\varepsilon^2\left(|z'|_{{}_{L^2}}^2
  +|\hat v|_{{}_{L^2}}^2+|w'|_{{}_{L^2}}^2\right)
\end{equation*}
\begin{equation*}
 \begin{aligned}
 |\Real\langle z',k\Theta_2'\rangle|
 &\leq  |\Real\langle z'',k\Theta_2\rangle| + |\Real\langle z',O(1)W'\Theta_2\rangle|\\
 &\leq C\varepsilon|z|_{{}_{W'}}^2+ C\varepsilon^2\bigl(|z'|_{{}_{L^2}}^2
  +|\hat v|_{{}_{L^2}}^2\bigr)+C|w'|_{{}_{L^2}}^2.
  \end{aligned}
\end{equation*}
Collecting all of these estimates, we complete the proof of 
Lemma \ref{lemma:friedRed}.
\end{proof}

\begin{proof}[Proof of Proposition \ref{prop:kawareduced}]
Choosing $\eta=\varepsilon$ and summing
up the estimates in Lemma \ref{lemma:friedRed}, we obtain,
for $\varepsilon$ small enough,
\begin{equation*}
   \Real\lambda\,|z|_{{}_{L^2}}^2+|z'|_{{}_{L^2}}^2
  \leq C\bigl(|z|_{{}_{W'}}^2
    +\varepsilon^2|\hat v|_{{}_{L^2}}^2
   +\varepsilon^{-1}\,|w'|_{{}_{L^2}}^2\bigr)
\end{equation*}
Using (in place of a first-order Friedrichs estimate) the bound
\begin{equation*}
  |\hat v|_{{}_{L^2}}^2+ |w'|_{{}_{L^2}}^2
\leq  C\left(|\Pi v|_{{}_{L^2}}^2+ |w'|_{{}_{L^2}}^2\right)
\leq C\varepsilon^2 |u|_{{}_{L^2}}^2
\end{equation*}
obtained in Proposition \ref{prop:kawa}, we get 
\begin{equation*}
  \Real\lambda\,|z|_{{}_{L^2}}^2+|u|_{{}_{L^2}}^2
  \leq C|z|_{{}_{W'}}^2+C\varepsilon |u|_{{}_{L^2}}^2.
 \end{equation*}
Hence estimate (\ref{kawared}) holds for $\varepsilon$ small.
\end{proof}

The reduced Kawashima estimate (\ref{kawared}) shows that it is possible to bound
the $L^2$ estimate of $u$ in term of $|z|_{{}_{W'}}$.
If we are able to prove a Poincar\'e--like inequality and bound 
the weighted norm $|z|_{{}_{W'}}$ by small multiples of the 
$L^2$ norm of $u$ and higher derivatives, we are done. 
This we can accomplish by changing variables in an appropriate way
and applying a weighted energy method in the spirit
of Goodman \cite{Good91,HuZu02,LiYu04}.
 
\begin{lemma}[ \cite{HuZu02} ] \label{special}
Let $a=a(W)$ and $b=b(W)$ be symmetric matrices, $b(W)\geq 0$,
with one eigenvalue $\alpha_0$ of $a$ close to zero and the others strictly
negative or positive (and uniformly separated from $\alpha_0$).
Then, there exist smooth, real matrix-valued
functions $r=r(x)$, $\ell=\ell(x)$, $\ell(x)\,r(x)=I$ for any $x$, 
satisfying, for some $C, c>0$, 
\begin{equation*}
(\ell\,r')_{pp}=0,\qquad\quad |\ell'|, |r'|\leq C|W'|;
\end{equation*}
\begin{equation}\label{blockdiag}
\ell\,a\,r=\diag(\alpha_-, \alpha_p, \alpha_+)
    =\left(\begin{array}{ccc} 
     \alpha_- & 0 & 0 \\ 0 & \alpha_p & 0  \\  0 & 0 & \alpha_+
     \end{array}\right);
\end{equation}
with $\alpha_p$ scalar, $\alpha_-, \alpha_+$ symmetric 
square matrices (with dimensions $p-1$ and
$n-p$ respectively),  $\alpha_-\leq -c<0<c\leq \alpha_+$; and
\begin{equation}\label{approxpos}
\Real \ell\, b\, r \ge -C\varepsilon.
\end{equation}
\end{lemma}

\begin{proof}
Since $a$ is symmetric, it is possible to find an othonormal transformation
$\omega=\omega(W)$ such that $\omega^t\,a\,\omega$ is (block-)diagonal with the
decomposition given in the righthand side of (\ref{blockdiag}).
The spectral separation assumption guarantees the positivity/negativity of
$\alpha_+/\alpha_-$.
Moreover, the matrix $\omega^t\,b\,\omega$ is 
positive semidefinite, $\Real b\ge 0$.

Let $\omega_p$ denote the $p$th column of $\omega$ and 
$\gamma=\gamma(x)$ be the solution of the first order linear
differential equation 
\begin{equation}\label{ode}
\gamma'= -(\omega_p\cdot\diff\omega_p\,W')\,\gamma, \quad \gamma(0)=1,
\end{equation}
or, equivalently, set
\begin{equation}\label{explicitgamma}
 \gamma(x):=\exp\left(\int_0^x \omega_p^t(W)\,\diff\omega_p(W)\,W'\,dy\right).
\end{equation}
Define the matrix $r$ and $\ell$ as 
\begin{equation}\label{defellr}
 r(x):=\omega(W)\,\diag(I_{p-1}, \gamma(x), I_{n-p}),\qquad
 \ell(x):=r^{-1}(x).
\end{equation}
Clearly estimates on $|r'|$ and $|\ell'|$ hold and
\begin{equation*}
 \begin{aligned}
 \ell\,a\,r&=\diag(I_{p-1}, \gamma^{-1}, I_{n-p})
 \diag(\alpha_-, \alpha_p, \alpha_+)
  \diag(I_{p-1}, \gamma, I_{n-p})\\
  &=\diag(\alpha_-, \alpha_p, \alpha_+),
  \end{aligned}
\end{equation*}
hence $\ell$ and $r$ still block-diagonalize $a$ in the manner claimed.
Moreover
\begin{equation*}
 (\ell r')_{pp}=\gamma^{-1}\omega_p \cdot (\gamma \omega_p)'
 =\omega_p \cdot \bigl(\diff\omega_p W'
 -(\omega_p\cdot \diff\omega_p\,W')\omega_p\bigr)=0,
\end{equation*}
since $\omega_p$ has norm equal to 1.

By (\ref{explicitgamma}), it follows that $\gamma=1+O(\varepsilon)$;
hence bound (\ref{approxpos}) follows from assumptions on $b$ and continuity.
\end{proof}

\begin{proof}[Proof of Proposition \ref{goodprop}]
To prove (\ref{good}), it is sufficient to establish the corresponding 
result for $\zeta:=\ell\,z$ with $\ell$ given in Lemma \ref{special}. 
Left multiplying (\ref{zequation}) by $\ell$, we get the eigenvalue
equation for $\zeta$
 \begin{equation}\label{goodmanzeta}
    \lambda\,\zeta+\alpha\,\ell\,r'\zeta+\alpha\,\zeta'-\beta\,\zeta''=\Xi,
  \end{equation}
where
 \begin{equation*}
   \alpha:=\ell\,a\,r,\qquad \beta:=\ell\,b_0\,r,\qquad \Xi:=\Xi_1+\Xi_2'
 \end{equation*}
 and
\begin{equation*}
 \left\{\begin{aligned}
 &\Xi_1:=O(1)|W'|\bigl(\varepsilon\,\zeta+\zeta'+\hat v'+w'\bigr),\\
 &\Xi_2:=O(1)|W'|\bigl(\zeta+\hat v\bigr)+O(1)\varepsilon\,\zeta'+O(1)\,w'.
 \end{aligned}\right.
\end{equation*}
Following \cite{Good91}, set $\rho_0(x):=1$ for any $x$, and define the two weights 
 $\rho_\pm$ as the solutions to the Cauchy problem
\begin{equation}
 \rho_\pm' = \mp M\,|W'|\,c^{-1}\,\rho_\pm, 
 \qquad \rho_\pm(0):=1, 
 \label{a4.23}
\end{equation}
where $c$ is given in Lemma \ref{special} 
and $M$ is a constant to be chosen later.
Therefore, for $\varepsilon$ so small that $O(M\varepsilon)<1$,
\begin{equation}
 \rho_\pm(x) = \exp\left(\pm\int^{x}_0 M|W'(\xi)|c^{-1}\,d\xi\right)
 = 1 + O\left(M\int_{\R}|W'(\xi)|\,d\xi\right)=O(1),
 \label{a4.24}
\end{equation}
and
\begin{equation}
 \rho'_j (x)= O(1)\,|W'(x)|, \qquad j\in\{-,0,+\}.
 \label{a4.25}
\end{equation}
Let $\rho=\rho(x)$ be the block diagonal matrix defined by
\begin{equation*}
  \rho(x):=\diag(\rho_-(x)\, I_h, \rho_0(x), \rho_+(x)\, I_k)
\end{equation*}
where $I_n$ denotes the identity $n\times n$ matrix.
Taking the real part of the $L^2-$scalar product of 
$\rho\,\zeta$ against (\ref{goodmanzeta}),  we get
 \begin{equation}\label{rhozeta}
   \Real\lambda\,\langle \rho\,\zeta,\zeta\rangle
    +\Real\langle\rho\,\zeta,\alpha\,\ell\,r' \zeta\rangle
   +\Real\langle \rho\,\zeta,\alpha\,\zeta'\rangle
    -\Real\langle \rho\,\zeta,\beta\,\zeta''\rangle
    =\Real\langle \rho\,\zeta,\Xi\rangle.
 \end{equation}
The weights $\rho_0,\rho_\pm$ are positive and $O(1)$, hence
$\langle \rho\,\zeta,\zeta\rangle^{1/2}$ is equivalent to $|\zeta|_{{}_{L^2}}$.
 
Both $\rho$ and $\rho\,\alpha$ are symmetric, hence
 \begin{equation*}
  \Real\langle \rho\,\zeta,\alpha\,\zeta'\rangle
  = \Real\langle \zeta,\rho\,\alpha\,\zeta'\rangle
  =-\frac12 \Real\langle \zeta,(\rho'\,\alpha+\rho\,\diff\alpha\,W')\,\zeta\rangle.
 \end{equation*}
By \eqref{gnl} and \eqref{dirder}, we have the key 
fact\footnote{Indeed, this is what drives the Goodman estimate; see \cite{Good86,Good91}.}
\begin{equation*}
 \diff\alpha_0\, W'\le -C |W'|
\end{equation*}
for some $C>0$.
By definition of $\rho_\pm$,  we have also
\begin{equation*}
   \rho'_\pm\,\alpha_\pm+\rho_\pm\,\diff\alpha_\pm\,W'
   =\mp\rho_\pm\,(M\,|W'|c^{-1}\alpha_\pm-\diff\alpha_\pm\,W').
\end{equation*}
Thus, for $M$ sufficiently large, there exists $C>0$, 
independent on $\varepsilon$, such that
\begin{equation*}
 \begin{aligned}
  \rho'\,\alpha+\rho\,\diff\alpha\,W'
   =&\, \diag( \rho'_-\,\alpha_-+\rho_-\,\diff\alpha_-\,W', 
    \diff\alpha_0\,W', \rho'_+\,\alpha_++\rho_+\,\diff\alpha_+\,W')\\
    &\leq -C\,|W'|\,\diag(M,1,M).
  \end{aligned}
\end{equation*}
Decomposing $\zeta$ as $(\zeta_-,\zeta_0,\zeta_+)$ and setting
 $\hat\zeta:=(\zeta_-,\zeta_+)$, we get the ``good'' term
\begin{equation*}
   \Real\langle \rho\zeta, \alpha\,\zeta'\rangle
   \geq C\int_{\R} (M|\hat \zeta|^2+|\zeta_0|^2)\,|W'|\,dx.
\end{equation*}
Next, let us deal with the ``bad'' term 
$\langle\rho\,\zeta,\alpha\,\ell\,r'\zeta\rangle$.
Since $(\ell\,r')_{pp}=0$, there holds
 \begin{equation*}
   |\Real\langle\rho\,\zeta,\alpha\,\ell\,r'\zeta\rangle|
   \leq C\int_{\R} |\hat\zeta|^2|W'|\,dx
 \end{equation*}
Hence, by choosing $M$ large enough, we get from (\ref{rhozeta})
\begin{equation*}
   \Real\lambda|\zeta|_{{}_{L^2}}^2+|\zeta|_{{}_{W'}}^2
    -\Real\langle \rho\,\zeta,\beta\,\zeta''\rangle
      \leq C\,\bigl|\Real\langle \rho\,\zeta,\Xi\rangle\bigr|
  \end{equation*}
Since
\begin{equation*}
 \Real\langle \rho\,\zeta,\beta\,\zeta''\rangle
 =-\Real\langle \rho\,\zeta',\beta\,\zeta'\rangle
 -\Real\langle \rho\,\zeta,\beta'\,\zeta'\rangle
 -\Real\langle \rho'\,\zeta,\beta\,\zeta'\rangle
\end{equation*}
the term with $\beta$ can be estimated by 
\begin{equation*}
 \Real\langle \rho\,\zeta,\beta\,\zeta''\rangle
 \leq C\varepsilon\bigl(|\zeta|_{{}_{W'}}^2
  +|\zeta'|_{{}_{L^2}}^2\bigr)
\end{equation*}
having used (\ref{approxpos}).
Hence, we obtain
\begin{equation}\label{abalone}
   \Real\lambda|\zeta|_{{}_{L^2}}^2+|\zeta|_{{}_{W'}}^2
      \leq C\varepsilon\bigl(|\zeta|_{{}_{W'}}^2
  +|\zeta'|_{{}_{L^2}}^2\bigr)
    +C\,\bigl|\Real\langle \rho\,\zeta,\Xi\rangle\bigr|
  \end{equation}
Given $\eta>0$, recalling (\ref{sizeder}), we deduce
\begin{equation*}
 \bigl|\Real\langle \rho\,\zeta,\Xi_1\rangle\bigr|
 \leq C(\varepsilon+\eta)|\zeta|_{{}_{W'}}^2
  +C\eta^{-1}\varepsilon^2(|\zeta'|_{{}_{L^2}}^2
  +|\hat v|_{{}_{L^2}}^2+|w'|_{{}_{L^2}}^2)
\end{equation*}  
with $C$ independent on $\eta$.
For what concerns the term with $\Xi_2$, integrating by parts
and using (\ref{a4.25}), there holds  
  \begin{equation*}
 \bigl|\Real\langle \rho\,\zeta,\Xi_2'\rangle\bigr|
 =\bigl|\Real\langle \rho'\,\zeta,\Xi_2\rangle\bigr|
+\bigl|\Real\langle \rho\,\zeta',\Xi_2\rangle\bigr|
 \leq \bigl|\Real\langle O(1)|W'|\,\zeta,\Xi_2\rangle\bigr|
+\bigl|\Real\langle \rho\,\zeta',\Xi_2\rangle\bigr|
 \end{equation*}  
For any $\eta>0$, estimating one by one the terms in $\Xi_2$,
we obtain
  \begin{equation*}
   \bigl|\Real\langle O(1)|W'|\,\zeta,\Xi_2\rangle\bigr|
   \leq C(\varepsilon+\eta) |\zeta|_{{}_{W'}}^2 
    +C\varepsilon^2(|\zeta'|_{{}_{L^2}}^2+|\hat v|_{{}_{L^2}}^2)
    +C\eta^{-1}\varepsilon^2 |w'|_{{}_{L^2}}^2,
 \end{equation*}  
    \begin{equation*}
 \bigl|\Real\langle \rho\,\zeta',\Xi_2\rangle\bigr|
  \leq C\eta|\zeta|_{{}_{W'}}^2
   +C(\varepsilon+\eta^{-1}\varepsilon^2+\eta)|\zeta'|_{{}_{L^2}}^2
   +C\eta^{-1}\varepsilon^2|\hat v|_{{}_{L^2}}^2+C\eta^{-1}|w'|_{{}_{L^2}}^2.
 \end{equation*} 
Choosing $\eta=\varepsilon$ and summing up,  we get
  \begin{equation*}
    \bigl|\Real\langle \rho\,\zeta,\Xi_2'\rangle\bigr|
    \leq C\varepsilon\bigl(|\zeta|_{{}_{W'}}^2 
   +|\zeta'|_{{}_{L^2}}^2+|\hat v|_{{}_{L^2}}^2\bigr)
   +C\varepsilon^{-1}|w'|_{{}_{L^2}}^2. 
 \end{equation*} 
Inserting these estimates in (\ref{abalone}), we get,
for $\varepsilon$ sufficiently small,
\begin{equation*}
   \Real\lambda|\zeta|_{{}_{L^2}}^2+|\zeta|_{{}_{W'}}^2
      \leq C\bigl(|\zeta'|_{{}_{L^2}}^2+\varepsilon|\hat v|_{{}_{L^2}}^2
   +\varepsilon^{-1}|w'|_{{}_{L^2}}^2\bigr).
  \end{equation*}
Since $\zeta=\ell\,z$, from the above 
estimate we deduce 
\begin{equation*}
   \Real\lambda|z|_{{}_{L^2}}^2+|z|_{{}_{W'}}^2
      \leq C\bigl(\varepsilon|z'|_{{}_{L^2}}^2
      +\varepsilon|\hat v|_{{}_{L^2}}^2
   +\varepsilon^{-1}|w'|_{{}_{L^2}}^2\bigr). 
  \end{equation*}
for $\varepsilon$ sufficiently small. 
Recalling that $z'=u$, estimate (\ref{good}) is proved.
\end{proof}

\appendix

\section{Structural hypotheses}\label{strucapp}

In this Appendix, we briefly discuss the structural hypotheses
of the introduction, verifying the assertions of Remark 
\ref{strucrmk} that {\sf A1--A2} and {\sf B1--B2} follow
from conditions (i)--(ii) of the remark 
(i.e., partial simultaneous symmetrizability plus genuine coupling)
together with the assumed structure $Q=(0_n,q)$.

\begin{lemma}\label{struclem}
Let $Q=(0_n,q)$.  Then, (i)--(ii) of Rmk.\ref{strucrmk} imply
{\sf A1--A2} and {\sf B1--B2}.
\end{lemma}

\begin{proof}
These follow by more general results of Yong \cite{Yong01}.\footnote{
Symmetrizability is not explicitly stated in \cite{Yong01}, but is clear
from the development.}
We give a proof for completeness.
As all properties are coordinate-independent properties of the linearization
about constant states, we may without loss
of generality take $A^0$ block-diagonal.
For, $TA^0$ is block-lower triangular for $T$ block-upper triangular,
whence $TA^0 T^*$ is symmetric block-diagonal, and a left symmetrizer
for the system obtained by the change of coordinates
$w\to (T^*)^{-1}w$, $A\to (T^*)^{-1}AT^*$, $Q\to (T^*)^{-1}QT^*$.

Observing that $\tilde A^0:=(A^0)^{-1}$ is a right symmetrizer if $A^0$ is
a left symmetrizer, we obtain
\begin{equation*}
\tilde A^0 w_t + \tilde A w_x= \tilde Q w,
\end{equation*}
where $\tilde A^0$ is symmetric positive definite and block-diagonal,
$\tilde A$ is symmetric, and 
$\diff \tilde Q= \begin{pmatrix}0 & 0\\ 0 & \tilde q\end{pmatrix}$ 
symmetric with $q<0$.
(Note: the latter key fact follows by 
\begin{equation*}
 \diff\tilde Q= (T^*)^{-1}\diff Q\,T^* \tilde A^0,
\end{equation*}
the fact that $T^*$ is block-lower triangular, and that the first block row
of $\diff Q$ by assumption vanishes.)
Rewriting, we have
\begin{equation*}
w_t + \bar A w_x= \bar Q w,
\end{equation*}
where 
\begin{equation*}
\bar A=\begin{pmatrix}
(\tilde A_{11}^0)^{-1}\tilde A_{11} &
(\tilde A_{11}^0)^{-1}\tilde A_{12}\\
(\tilde A_{22}^0)^{-1}\tilde A_{21} &
(\tilde A_{22}^0)^{-1}\tilde A_{22}
\end{pmatrix},\qquad
\diff\bar Q=\begin{pmatrix}
0& 0\\
0& (\tilde A_{22}^0)^{-1}\tilde q
\end{pmatrix}.
\end{equation*}

In these coordinates, one readily computes that 
\begin{equation*}
a_0=(\tilde A^0_{11})^{-1} \tilde A_{11},\qquad
b_0=-(\tilde A^0_{11})^{-1} \tilde A_{12}^*\tilde q \tilde A_{12}^*,
\end{equation*}
hence $\ker b_0= \ker \tilde A_{12}$, and genuine coupling for
the reduced system is the condition that no eigenvector
of $a_0=(\tilde A^0_{11})^{-1} \tilde A_{11}$ lie in $\ker \tilde A_{12}$,
the same condition as for genuine coupling of the full system,
and $\tilde A^0_{11}$ is a left symmetrizer for the reduced system
with  $\tilde A^0_{11}b_0= -\tilde A_{12}^*\tilde q \tilde A_{12}^*$
symmetric positive semidefinite since $\tilde q$ is 
symmetric negative definite.
\end{proof}

\end{document}